\documentclass{amsart}
\usepackage[english]{babel}
\usepackage{amsfonts, amsmath, amsthm, amssymb,amscd,indentfirst}

\usepackage{esint}
\usepackage{graphicx}
\usepackage{xcolor}
\usepackage{MnSymbol}
\usepackage{palatino}

\newtheorem{theorem}{Theorem}[section]

\newtheorem{proposition}{Proposition}[section]

\newtheorem{definition}{Definition}[section]

\newtheorem{corollary}{Corollary}[section]

\newtheorem{example}{Example}[section]
\newtheorem{remark}{Remark}[section]

\begin{document}

\title[A Feynman-Kac formula]{A Feynman-Kac formula for differential forms on manifolds with boundary and  applications}
\author{Levi Lopes de Lima}
\address{Universidade Federal do Cear\'a,
	Departamento de Matem\'atica, Campus do Pici, R. Humberto Monte, s/n, 60455-760,
	Fortaleza/CE, Brazil.}
\email{levi@mat.ufc.br}

\begin{abstract}
We prove a Feynman-Kac formula for differential forms satisfying absolute boun\-dary conditions on Riemannian manifolds with boundary and of bounded geometry.
We use this to construct $L^2$ harmonic forms out of bounded ones on the universal cover of a compact Riemannian manifold whose geometry displays a positivity property expressed in terms of a certain stochastic average of the Weitz\-enb\"ock operator $R_p$ acting on $p$-forms and the second fundamental form of the boundary. This extends previous work by Elworthy-Li-Rosenberg on closed manifolds to this setting. As an
application we find a geometric obstruction to the existence of  metrics with 2-convex boundary and positive $R_2$ in this stochastic sense.  We also discuss a version of the Feynman-Kac  formula for spinors under suitable boundary conditions. 
\end{abstract}

\maketitle

\section{Introduction}\label{intro}

A celebrated result by Gromov \cite{Gr}  says that an open manifold carries metrics with positive and metrics with negative sectional curvature. Thus, in any such manifold there is enough room to interpolate between two rather disparate types of geometries.
In contrast, no such flexibility is available in the context of closed manifolds. For instance, it already follows from Hadamard and Bonnet-Myers theorems from basic Riemannian Geometry that a closed manifold which carries a metric with non-positive sectional curvature does not carry a metric with positive Ricci curvature. 

Our interest here lies in another manifestation of this ``exclusion principle'' for closed manifolds due to Elworthy-Li-Rosenberg \cite{ELR}. 
Relying heavily on stochastic methods, these authors put forward  an elegant refinement of the famous Bochner technique with far-reaching consequences. For example, they prove that a sufficiently negatively pinched closed  manifold does not carry a metric whose Weitzenb\"ock operator acting on $2$-forms is even allowed to be negative in  a region of small volume, an improvement which definitely turns the obstruction unapproachable by the classical reasoning \cite{R}. 
When trying to extend this kind of geometric obstruction to compact manifolds with boundary ($\partial$-manifolds, for short) we should have in mind that balls carry a huge variety of metrics as illustrated by geodesic balls in an arbitrary Riemannian manifold. These examples also show that the boundary can always be chosen convex just by taking the  radius  sufficiently small. Thus, even if we insist on having the boundary appropriately convex in both metrics, some topological assumption on the underlying manifold must be imposed.      
The purpose of this note is to present results in this direction which qualify as natural extensions of those in \cite{ELR}.

If $N$ is a Riemannian $\partial$-manifold of dimension $n$, the   Weitzenb\"ock decomposition reads
\[
\Delta_q=\Delta^B_q+R_q,
\]
where $\Delta_q=dd^\star+d^\star d$ is the Hodge Laplacian acting on $q$-forms, $1\leq q\leq n-1$, $d^\star=\pm\star d\star$ is the co-differential, $\star$ is the Hodge star operator,  $\Delta^B_q$ is the Bochner Laplacian and $R_q$, the Weizenb\"ock curvature operator, depends linearly on the curvature tensor, albeit in a rather complicated way. We remind that $R_1={\rm Ric}$, and since $\star R_p=R_{n-p}\star$, this also determines $R_{n-1}$,  but in general the structure of $R_q$, $2\leq q\leq n-2$, is notoriously hard to grasp.
To these invariants we attach the functions $r_{(q)}:N\to\mathbb R$, $r_{(q)}(x)=\inf_{|\omega|=1}\langle R_q(x)\omega,\omega\rangle$, the least eigenvalue of $R_q(x)$. 
We also consider the  principal curvatures $\rho_1,\cdots ,\rho_{n-1}$ of $\partial N$ computed  with respect to the inward unit normal vector field. For each $x\in\partial N$ and $q=1,\cdots,n-1$, define
\[
\rho_{(q)}(x)=\inf_{1\leq i_1<\cdots<i_q\leq n-1}\rho_{i_1}(x)+\cdots +\rho_{i_q}(x),
\]
the sum of the $q$ smallest principal curvatures at $x$.
We say that $\partial N$ is $q${\rm -convex} if $\underline\rho_{(q)}: =\inf_{x\in\partial M}\rho_{(q)}(x)>0$. Note that $q$-convexity implies $(q+1)$-convexity. Also, $N$ is said to be convex if $\underline{\rho}_{(1)}\geq 0$ everywhere. Finally, recall that a Riemannian metric $h$ on a manifold is $\kappa$-negatively pinched if its sectional curvature satisfies 
$-1\leq K_{\rm sec}(h)< -\kappa<0$.

Stochastic  notions make their entrance in the theory by means of the following considerations. Let $N$ be a Riemannian $\partial$-manifold. In case $N$ is non-compact we always assume that the underlying metric $h$ is complete and the triple $(N,\partial N,h)$ has bounded geometry in the sense of \cite{S1,S2,S3}. We then consider   
reflecting Brownian motion $\{x^t\}$ on $N$ starting at some $x^0\in N$; see Section \ref{appA} for a (necessarily brief) description of this diffusion process. Let $\alpha:N\to\mathbb R$ and $\beta:\partial N\to \mathbb R$ be $C^1$ functions. Adapting a classical definition to our setting, we say that the pair $(\alpha,\beta)$ is  strongly stochastically positive (s.s.p.) if 
\[
{{\limsup}}_{t\to+\infty}\frac{1}{t}\sup_{x^0\in K}\log\mathbb E_{x^0}\left(\exp\left({-\frac{1}{2}\int_0^{t}\alpha(x^s)ds}-\int_0^{t}\beta(x^s)dl^s\right)\right)<0,
\]
for any $K\subset N$ compact, where $l^t$ is the boundary local time associated to $\{x^t\}$. This is certainly the case if both $\alpha$ and $\beta$ have strictly positive lower bounds but the point to emphasize here is that, at least if $N$ is compact, it might well happen with the functions being positive except possibly in regions of small volume, given that the definition involves expectation with respect to the underlying diffusion.

Similarly to \cite{ELR}, our main results  provide examples of $\partial $-manifolds for which there holds an exclusion principle involving the various notions of curvature appearing above. From now on we  always assume that $n\geq 4$ and set $\kappa_p=p^2/(n-p-1)^2$.

\begin{theorem}
	\label{mainesp}
	Let $M$ be  a compact $\partial$-manifold with infinite fundamental group. Assume also that $M$ satisfies $H^p(M;\mathbb R)\neq 0$, where $2\leq p< (n-1)/2$. If $M$ carries a convex $\kappa_p$-negatively pinched metric then it does not carry a metric  with both $(r_{(p\pm 1)},\rho_{(p-1)})$ s.s.p.  
\end{theorem}

A somewhat more satisfactory  statement is available for $p=1$.

\begin{theorem}
	\label{mainesp2} Let $M$ be a compact manifold with nonamenable fundamental group. If $M$ carries a convex $\kappa_1$-negatively pinched metric then it does not carry a metric with $(r_{(2)},\rho_{(2)})$ s.s.p.
\end{theorem}

\begin{remark}
	\label{ssp}{\rm  These results correspond respectively to Corollary 2.1 and Theorem 2.3 in \cite{ELR}.  We point out that our  assumptions on the fundamental group are natural in the sense that they are  automatically satisfied there. As mentioned above, balls are obvious counterexamples to our results  if the topological assumptions are removed. Also, the manifold $\mathbb S^1\times\mathbb D^{n-1}$ shows that merely assuming that the fundamental group is infinite does not suffice in Theorem \ref{mainesp2}; see Remark \ref{countex} below. On the other hand, it is not clear whether the convexity hypothesis with respect to the negatively curved metric can be relaxed somehow.}  
\end{remark}

Using Theorem \ref{mainesp2} we can exhibit an interesting  class of    compact $\partial$-manif\-olds for which a natural class of metrics is excluded.

\begin{theorem}
\label{hypno} 
If $X$ is a closed hyperbolic manifold of dimension $l\geq 2$ then its product with a disk $\mathbb D^m$ does not carry a metric with  $(r_{(2)},\rho_{(2)})$ s.s.p. 
\end{theorem}

\begin{proof}
	Write $X=\mathbb H^l/\Gamma$ as the quotient of hyperbolic space $\mathbb H^l$ by a (necessarily nonamenable) group $\Gamma$ of hyperbolic motions. Embed $\mathbb H^l$ as a totally geodesic submanifold of $\mathbb H^{l+m}$ and let $\widetilde M\subset\mathbb H^{l+m}$ be a tubular neighborhood of $\mathbb H^l$ of constant radius. Extend the $\Gamma$-action to $\widetilde M$ in the obvious manner and observe that, since  $\widetilde M$ is convex, $M=\widetilde M/\Gamma=X\times\mathbb D^m$ with the induced hyperbolic metric is convex as well. Thus,  Theorem \ref{mainesp2} applies.
\end{proof}

\begin{remark}
	\label{unobst} {\rm Theorem \ref{hypno} provides a geometric  obstruction to the existence of metrics with $(r_{(2)},\rho_{(2)})$ s.s.p. Notice that if the second Betti number of $X$ vanishes, the obstruction can not be detected by the classical version of the Bochner technique for $\partial$-manifolds \cite[Chapter 8]{Y} even if we assume strict positivity of $(r_{(2)},\rho_{(2)})$.}
\end{remark}

\begin{remark}
	\label{geodcy}{\rm A larger class of manifolds to which the conclusion of Theorem \ref{hypno} obviously holds is formed by tubular neighborhoods of closed embedded totally geodesic submanifolds in a given hyperbolic manifold.}
\end{remark}

\begin{corollary}\label{hypnocor}
	Under the conditions of Theorem \ref{hypno}, assume that $n=l+m$ is even. Then $X\times\mathbb D^m$ does not carry a metric with positive isotropic curvature and $2$-convex boundary. 	
\end{corollary}

\begin{proof}
	For even dimensional manifolds it is shown in \cite{MW} that positive isotr\-opic curvature implies $R_2>0$.
\end{proof}

\begin{remark}
	\label{anew}{\rm Since the computation in \cite{MW} expresses $R_2$ as a sum of isotropic curvatures, in Corollary \ref{hypnocor} we can even relax the condition on the metric to allow  the invariants to be negative in a  region of small volume. }
\end{remark}

\begin{remark}
	\label{countex}{\rm 
	The standard product metric on $\mathbb S^1\times\mathbb S^{n-1}$ is known to have positive isotropic curvature. It is easy to check that if $r<\pi/2$ the boundary of the tubular neighborhood $U_r\subset\mathbb S^1\times\mathbb S^{n-1}$ of radius $r$ of the circle factor is $2$-convex. Thus, the conclusion of Corollary  \ref{hypnocor} does not hold for $U_r=\mathbb S^1\times\mathbb D^{n-1}$. Notice that $U_r$ carries a convex hyperbolic metric since its universal cover $\widetilde U_r=\mathbb R\times\mathbb D^{n-1}$ is diffeomorphic to a tubular neigborhood of a geodesic in $\mathbb H^n$. The problem here is that the fundamental group is abelian, hence amenable,  and the argument leading to Theorem \ref{mainesp2} breaks down. This  also can be understood in stochastic terms. In effect, the proof of Theorem \ref{mainesp2} shows that Brownian motion on the universal cover is transient, while recurrence certainly occurs in $\widetilde U_r$; see Remark \ref{functor}. In this respect it would be interesting to investigate if the conclusion of Theorem  \ref{hypno} holds in case $X$ is  flat or, more generally, has non-positive sectional curvature.}
\end{remark}

\begin{remark}
	\label{fraser}{\rm Compact $\partial$-manifolds with positive isotropic curvature have des\-erved a lot
		of attention in recent years. An important result by Fraser \cite{F} says  that such a $\partial$-manifold  is contractible if it is simply connected and  its boundary is connected and $2$-convex. The proof combines index estimates for minimal surfaces and a variant of the Sachs-Uhlenbeck theory adapted to this setting. However, as the examples in Remark \ref{countex} testify, this geometric condition is compatible with an infinite fundamental group. With no assumption on the fundamental group or on the topology of the boundary, these techniques still imply that  
		all the (absolute and relative) homotopy groups vanish in the range $2\leq i\leq n/2$. Moreover, it is shown in \cite{CF} that the fundamental group of the boundary injects into the fundamental group of the manifold. 
		However, if we take $m\geq l+2$ it is easy to check that none of these homotopical obstructions rules out  the  metrics in Corollary \ref{hypnocor}. We point out that a conjecture in \cite{F} asserts that a closed, embedded  $2$-convex hypersurface in a manifold with positive isotropic curvature is either $\mathbb S^n$ or a connected sum of finitely many copies of $\mathbb S^1\times \mathbb S^{n-1}$. Since the fundamental group of a closed hyperbolic manifold is neither infinite cyclic nor a free product, Corollary \ref{hypnocor} provides further support to the conjecture.}
\end{remark}

This paper is mostly inspired on the beautiful work by Elworthy-Rosenbeg-Li \cite{ELR}. Their ideas are used in Section \ref{setup} to construct $L^2$ harmonic forms on the universal cover of certain compact $\partial$-manifolds  starting from bounded ones. This is precisely where stochastic techniques come into play and a crucial ingredient at this point is a Feynman-Kac-type formula for differential forms in higher degree meeting absolute boundary conditions. In order not to interrupt the exposition, this technical result is established in the final Chapter \ref{appA} following ideas in \cite{H1}, where the case of $1$-forms is treated; see also \cite{A,IW} for previous contributions. To illustrate the flexibility of the method we also discuss a similar formula for spinors evolving under the heat semigroup generated by the Dirac Laplacian on a ${\rm spin}^c$ $\partial$-manifold under suitable boundary conditions. Another important ingredient in the argument is a Donnelly-Xavier-type eigenvalue estimate described in Section \ref{dxest}, whose proof   uses both the convexity and the assumption that the  fundamental group is infinite. Combined with Schick's $L^2$ Hodge-de Rham theory \cite{S2,S3} this allows us to prove a vanishing result for the relevant $L^2$ cohomology group. Finally, the proofs  of the main applications (Theorem \ref{mainesp} and \ref{mainesp2} above) are presented in Section \ref{proofsmain}. 

\vspace{0.3cm}
\noindent
{\bf Acknowledgements.} The author would  like to thank Rafael Montezuma (Prin\-ceton University) for conversations.

\section{From bounded to $L^2$ harmonic forms}\label{setup}

We consider a complete Riemannian $\partial$-manifold $N$ with boundary $\partial N$ oriented by an inward  unit normal vector field $\nu$. As always we assume that the triple $(N,\partial N,h)$ has bounded geometry in the sense of \cite{S1,S2,S3}. 
For us the case of interest occurs when $N=\widetilde M$, the universal cover of a compact $\partial$-manifold $(M,g)$ and $h=\widetilde g$, the lifted metric. Recall that a $q$-form $\omega$ on $N$ satisfies  {absolute} boundary conditions if 
\begin{equation}\label{abscond}
\nu\righthalfcup\omega=0,\quad \nu\righthalfcup d\omega=0
\end{equation}
along $\partial N$. Equivalently, 
\begin{equation}
\label{abascondeq}\omega_{\rm nor}=0,\quad(d\omega)_{\rm nor}=0, 
\end{equation}
where $\omega=\omega_{\rm tan}+\nu\wedge\omega_{\rm nor}$ is the  natural decomposition of $\omega$ in its tangential and normal components. Here, we identify $\nu$ to its dual $1$-form in the standard manner.
For simplicity we say that $\omega$ is {absolute} if any of these conditions is satisfied.
Notice that for $q=0$ this means that the given function satisfies Neumann boundary condition. 

For $t> 0$ let $P_t=e^{-\frac{1}{2}t\Delta^{\rm abs}_q}$ be the corresponding heat kernel acting on forms. Thus, for any absolute  $q$-form $\omega_0\in L^2\cap L^\infty$, $\omega_t=P_t\omega_0$  is a solution to the initial-boundary value problem
\begin{equation}
\label{bvp}
\frac{\partial\omega_t}{\partial t}  +  \frac{1}{2}\Delta^{\rm abs}_q\omega_t=0, \quad
\lim_{t\to 0} \omega_t =  \omega_0,\quad
\nu\righthalfcup\omega_t  =  0,\quad
\nu\righthalfcup d\omega_t  =  0.
\end{equation}
A proof of these facts follows from standard functional calculus and the elliptic machinery developed in \cite{S2} \cite{S3}.

A key ingredient in our approach is a Feynman-Kac-type representation of any solution $\omega_t$ as above in terms of  Brownian motion in $N$. This is well-known to hold in the boundaryless case \cite{E, H2, Gu, Ma, St}.  However, as  pointed out in \cite{H1}, where the case $q=1$ is discussed in detail, extra difficulties appear when trying to establish a similar result in the presence of a boundary. 
In Section \ref{appA} we explain how the method in \cite{H1} can be adapted to establish a Feynman-Kac formula for solutions of (\ref{bvp}), regardless of the value of $q$; see Theorem \ref{fkac}. For the moment  we need an immediate consequence of this formula, namely, the useful estimate 
\begin{equation}\label{stocest2}
|\omega_t(x^0)|\leq \mathbb E_{x^0}\left(|\omega_0(x^t)|\exp\left(
-\frac{1}{2}\int_0^t r_{(q)}(x^s)ds-\int_0^t\rho_{(q)}(x^s)dl^s
\right)\right),
\end{equation}
where $\{x^t\}$
is  reflecting Brownian motion on $N$ starting at $x^0$ and  
$l^t$ is the associated boundary local time. 
The remarkable feature of (\ref{stocest2}) is that  the geometric invariants $r_{(q)}$ and $\rho_{(q)}$ play entirely similar roles in stochastically controlling the solution in the long run. 
Now we put this estimate to good use and establish a central result in this work; compare to \cite[Lemma 2.1]{ELR}. 

\begin{proposition}\label{keyprop}
Set $P=\lim_{t\to+\infty}P_t$,
\[
\theta_{q}(x^0)=\int_0^{+\infty}\mathbb E_{x^0}\left(
\exp\left(-\frac{1}{2}\int_0^tr_{(q)}(x^s)ds-\int_0^t\rho_{(q)}(x^s)dl^s\right)
\right)dt,
\]
and take compactly supported $p$-forms $\phi$ and $\psi$ with $\psi_{\rm nor}=0$ along $\partial N$ and  $\phi=0$ in a neighborhood of $\partial N$. If $2\leq p\leq n-2$ there holds 
\begin{eqnarray*}
	\left|\int_N\langle P\phi-\phi,\psi\rangle dN\right| & \leq &
	\frac{1}{2}
	\left(\sup_{x^0\in\,{\rm supp}\,\phi}\theta_{p+1}(x^0)\right)|d\psi|_\infty|d\phi|_1\\
	& & \quad + \frac{1}{2}
	\left(\sup_{x^0\in\,{\rm supp}\,\phi}\theta_{p-1}(x^0)\right)|d^\star\psi|_\infty|d^\star \phi|_1.
\end{eqnarray*}
If $p=1$ we have instead
\begin{eqnarray*}
	\left|\int_N\langle P\phi-\phi,\psi\rangle dN\right| & \leq &
	\frac{1}{2}
	\left(\sup_{x^0\in\,{\rm supp}\,\phi}\theta_{2}(x^0)\right)|d\psi|_\infty|d\phi|_1\\
	& & \quad + \frac{1}{2}
	\sup_{x^0\in\,{\rm supp}\,\phi}\left|\int_0^{+\infty}(P_\tau d^\star\phi)(x^0)\,d\tau\right||d^\star\psi|_\infty.
\end{eqnarray*}	
\end{proposition}	

\begin{proof}
We have 
\begin{eqnarray*}
\int_N\langle P\phi-\phi,\psi\rangle dN & = & \lim_{t\to +\infty}\int_N\langle P_t\phi-P_0\phi,\psi\rangle dN\\
& = & \lim_{t\to+\infty}\int_0^t\int_N\langle\partial_\tau P_\tau\phi,\psi\rangle dNd\tau\\
& = & -\frac{1}{2}\lim_{t\to+\infty}\int_0^t\int_N\langle\Delta^{\rm abs}_p P_\tau\phi,\psi\rangle dNd\tau\\
& = & -\frac{1}{2}\lim_{t\to+\infty}\int_0^t\int_N\langle d P_\tau d^\star\phi,\psi\rangle dNd\tau\\
&  & \quad   -\frac{1}{2}\lim_{t\to+\infty}\int_0^t\int_N\langle d^\star P_\tau d\phi,\psi\rangle dNd\tau.
\end{eqnarray*}
We now recall Green's formula: if  $\alpha\wedge\star\beta$ is compactly supported then 
\[
\int_N \langle d\alpha,\beta\rangle dN=\int_N\langle \alpha,d^\star\beta\rangle dN+\int_{\partial N}\langle \alpha_{\rm tan},\star\beta_{\rm nor}\rangle d\partial N.
\]
Since  $(P_\tau d\phi)_{\rm nor}=0$ this leads to   
\begin{eqnarray*}
	\int_N\langle P\phi-\phi,\psi\rangle dN & =  & -\frac{1}{2}\lim_{t\to+\infty}\int_0^t\int_N\langle P_\tau d^\star\phi,d^\star\psi\rangle dNd\tau\\
	&  & \quad - \frac{1}{2}\lim_{t\to+\infty}\int_0^t\int_N\langle   P_\tau d\phi,d\psi\rangle dNd\tau.
\end{eqnarray*}
The result now follows by applying (\ref{stocest2})
to $\omega_\tau=P_\tau d^\star\phi$ and $\omega_\tau=P_\tau d\phi$. 
\end{proof}

From this we derive the existence of absolute $L^2$ harmonic $p$-forms from bounded ones under appropriate positivity assumptions; compare to \cite[Theorem 2.1]{ELR}. In the following we denote by $H^q_{(2),\rm abs}(N,h)$ the  $q^{\rm th}$ $L^2$ absolute cohomology group of $(N,h)$. We refer to \cite{S2,S3} for the definition and basic properties of these invariants, including the corresponding $L^2$ Hodge-de Rham theory.

\begin{proposition}
	\label{import}
	Let $(N,h)$ and $p$ be as above. Assume that both $\sup_{x^0\in K}\theta_{p+1}(x^0)$ and $\sup_{x^0\in K}\theta_{p-1}(x^0)$ are finite if $2\leq p \leq n-2$ and that both  $\sup_{x^0\in K}\theta_{2}(x^0)$ and $\sup_{x^0\in K}|\int_0^{+\infty}(P_\tau d^\star \phi)(x^0)d\tau|$ are finite if $p=1$, where   
	$K\subset N$ is any compact. Then $N$ carries a non-trivial absolute $L^2$ harmonic $p$-form whenever it carries a non-trivial absolute bounded harmonic $p$-form. In particular, $H^p_{(2),\rm abs}(N,h)$ is non-trivial.
\end{proposition}

\begin{proof}
Let $\psi$ be a non-trivial absolute bounded harmonic $p$-form. Consider a Gaff\-ney-type cutoff sequence
$\{h_n\}$ and set $\psi_n=h_n\psi$, so that each $\psi_n$ is compactly supported. Also, $\psi_n\to \psi$ and  $|d\psi_n|_\infty+|d^\star\psi_n|_\infty\to 0$  as $n\to +\infty$. Applying Proposition \ref{keyprop} with $\psi$ replaced by $\psi_n$ and sending $n\to +\infty$ we see that 
\[
\int_N\langle P\phi-\phi,\psi\rangle dN=0.
\]
If no non-trivial absolute $L^2$ harmonic $p$-form   exists then $P\phi=0$ for any $\phi$ and hence $\psi=0$,  a contradiction. The last assertion follows from the $L^2$ Hodge-de Rham theory in \cite{S2,S3}.
\end{proof}

\section{A Donnelly-Xavier-type estimate for $\partial$-manifolds}\label{dxest}

In this section we present a Donnelly-Xavier-type estimate for the universal cover of $\kappa$-negatively pinched 
$\partial$-manifolds which implies the vanishing of certain absolute  $L^2$ cohomology groups. This extends to this setting a sharp result for boundaryless manifolds obtained in \cite{ER1}, which by its turn improves on the original result in \cite{DX}. The exact analogue for $\partial$-manifolds of the estimate in \cite{DX}, hence with a tighter pinching, appears in \cite{S3}; see Remark \ref{schick} below. Our proof adapts a computation in \cite[Section 5]{BB}, where the sharp result for boundaryless manifolds is also achieved, and relies on a rather general integral formula.

\begin{proposition}
	\label{intfor}Let $(N,h)$ be a $\partial$-manifold, $f:N\to\mathbb R$ a $C^2$ function and $\{\mu_i\}_{i=1}^n$ the eigenvalues of the Hessian operator of $f$. If $p\geq 1$ then for any compactly supported $p$-form $\omega$ in $N$ there holds
	\begin{eqnarray*}
		\int_N\left(\langle d\omega,{\nabla f}\wedge\omega\rangle+\langle d^\star\omega,{\nabla f}\righthalfcup\omega\rangle\right) dN &=&\int_N\left(\sum_i\mu_i|{e_i}\righthalfcup\omega|^2+\frac{1}{2}|\omega|^2\Delta_0 f\right)dN\\		& & \quad -\int_{\partial N}\langle {\nabla f}\righthalfcup\omega,\nu\righthalfcup\omega\rangle d\partial N \\
		& & \qquad -\frac{1}{2}\int_{\partial N}|\omega|^2\langle \nabla f,\nu\rangle d\partial N,
	\end{eqnarray*}
	where $\{e_i\}$ is a local orthonormal frame diagonalizing the Hessian of $f$, $\{\mu_i\}$ are the corresponding eigenvalues and $\nu$ is the inward unit normal vector field along $\partial N$.
\end{proposition}

\begin{proof}
Consider the vector field $V$ defined by $\langle V,W\rangle=	\langle {\nabla f}\righthalfcup\omega,W\righthalfcup\omega\rangle $, for any $W$. 
A computation in \cite[Section 5]{BB} gives 
\[
{\rm div}\,V  = \sum_i\mu_i|e_i\righthalfcup \omega|^2   -\langle d\omega,\nabla f\wedge \omega\rangle -\langle{d^\star\omega,\nabla f}\righthalfcup\omega\rangle+\langle\nabla_{\nabla f}\omega,\omega\rangle.
\]	              
Integrating by parts we obtain
\begin{eqnarray*}
	\int_N\left(\langle d\omega,{\nabla f}\wedge\omega\rangle+\langle d^\star\omega,{\nabla f}\righthalfcup\omega\rangle\right) dN &=&\int_N\left(\sum_i\mu_i|{e_i}\righthalfcup\omega|^2+\langle\nabla_{\nabla f}\omega,\omega \rangle\right)dN\\		
	& & \quad -\int_{\partial N}\langle {\nabla f}\righthalfcup\omega,\nu\righthalfcup\omega\rangle d\partial N. 
\end{eqnarray*}
Now observe that
\begin{eqnarray*}
	\int_N\left(\langle\nabla_{\nabla f}\omega,\omega\rangle-\frac{1}{2}|\omega|^2\Delta_0  f\right)dN 
	& = & \frac{1}{2}\int_N\left(\langle\nabla f,\nabla |\omega|^2\rangle-|\omega|^2\Delta_0  f\right)dN \\
	& = & \frac{1}{2}\int_N{\rm div}(|\omega|^2\nabla f)dN\\
	& = & -\frac{1}{2}\int_{\partial N}|\omega|^2\langle\nabla f,\nu\rangle d\partial N.
\end{eqnarray*}
This completes the proof.
\end{proof}

We can now present a version of the Donnelly-Xavier-type estimate that suffices for our purposes.

\begin{proposition}
	\label{dxesttheo}
	Let $(M,g)$ be a compact and convex $\partial$-manifold with infinite fundamental group and assume that  $g$ satisfies $-1\leq K_{\rm sec}(g)\leq-\kappa<0$. If $p\geq 1$ then for any compactly supported $p$-form $\omega$ in $\widetilde M$ satisfying $\nu\righthalfcup \omega=0$ along $\partial\widetilde M $ there holds
	\begin{equation}\label{dx1}
	|d\omega|_2+|d^{\star}\omega|_2\geq\frac{1}{2}((n-p-1)\sqrt{\kappa}-p)|\omega|_2.
	\end{equation}
\end{proposition}

\begin{proof}
Convexity implies that any $x\in M\setminus \partial M$ and $y\in M$ can be joined by a minimizing geodesic segment lying in the interior of $M$ (except possibly for $y$). The same holds in $\widetilde M$ with the segment now being unique. Thus, for any $x\in \widetilde M\setminus\partial \widetilde M$ the Riemannian distance $d_x$ to $x$ is well-defined. Notice that  $\langle \nabla d_x,\nu\rangle\leq 0$ along $\partial\widetilde N$, $|\nabla d_x|=1$ and $\Delta_0 d_x=-\sum_i\mu_i$, where we may assume that $\mu_1=0$ .  Thus, using the  boundary condition  $\nu\righthalfcup\omega=0$ and  Proposition \ref{intfor} with $f=d_x$ we obtain 
\[
|\omega|_2(|d\omega|_2+|d^{\star}\omega|_2)\geq \int_N\left(\sum_i\mu_i|{e_i}\righthalfcup\omega|^2-\frac{1}{2}|\omega|^2\sum_i\mu_i\right)d\widetilde M.
\]
Expand $\omega=\sum_I\omega_Ie_I$, where $I=\{i_1<\cdots <i_p\}$ and $e_I=e_{i_1}\wedge\cdots\wedge e_{i_p}$. Since $\sum_i\mu_i|{e_i}\righthalfcup e_I|^2=\sum_{i\in I}\mu_i$ the right-hand side equals
\[
\frac{1}{2}\int_N\sum_{i,I}\left(\sum_{i\notin I}\eta_i-\sum_{i\in I}\eta_i\right)|\omega_I|^2d\widetilde M,
\]
where $\eta_i=-\mu_i$ are the principal curvatures of the geodesic ball centered at $x$.
Thus,  by standard comparison theory this is bounded from below by
\[
\frac{1}{2}\int_N\left((n-p-1)\sqrt{\kappa}\coth\sqrt{\kappa} d_x-p\coth d_x\right)|\omega|^2d\widetilde M.
\]
Now observe that $\widetilde M$ has infinite diameter because $\pi_1(M)$ is infinite. Hence, we can find a sequence $\{x_i\}\subset\widetilde M$ so that $d_{x_i}(y)\to +\infty$ uniformly in $y\in{\rm supp}\,\omega$. By taking $x=x_i$ and passing to  the limit we  obtain the desired result. 
\end{proof}

\begin{remark}
	\label{schick}{\rm Notice that (\ref{dx1}) is meaningful only if $\kappa>\kappa_p$, which forces $\kappa_p<1$, that is, $2p<n-1$. We note that under the conditions above it is proved in \cite{S3}  that 
		\[
		|d\omega|_2+|d^{\star}\omega|_2\geq\frac{1}{2}((n-1)\sqrt{\kappa}-2p)|\omega|_2.
		\]
		This only makes sense if $\kappa>\kappa_p':=4p^2/(n-1)^2$, which again forces $2p<n-1$, but notice that   (\ref{dx1}) gives a better pinching constant if $1\leq p<(n-1)/2$. It is observed in \cite{S3} that 
		\[
		|d\eta|_2+|d^{\star}\eta|_2\geq\frac{1}{2}((n-1)\sqrt{\kappa}-2(n-p))|\eta|_2,
		\]
		for any $p$-form $\eta$ satisfying $\nu\wedge\eta=0$ along $\partial\widetilde M$. Taking $p=n$ and using Hodge duality this means that 
		\begin{equation}\label{schickest}
		|d\varphi|_2\geq\frac{1}{2}(n-1)\sqrt{\kappa}|\varphi|_2,
		\end{equation}
		for any compactly supported function $\varphi$ satisfying Neumann boundary condition. In other words, (\ref{dx1}) holds for $p=0$ as well. This transplants to our setting a famous estimate by McKean \cite{Mc}. Observe however that the assumption on the fundamental
		group is essential in (\ref{schickest}) as  the first Neumann eigenvalue
		of geodesic balls in hyperbolic space converges to zero as the radius goes to
		infinity \cite{C}. Thus, (\ref{schickest}) illustrates a situation where a topological condition on a compact $\partial$-manifold
		poses spectral constraints on its universal cover.
		}
		\end{remark}

With these estimates at hand it is rather straightforward  to establish vanishing theorems for $L^2$ harmonic forms. For this we  consider $(M,g)$ as in Theorem \ref{dxesttheo} and define the absolute Hodge Laplacian $\Delta_p^{\rm abs}$ on $\widetilde M$ with domain  $\mathcal D(\Delta_p^{\rm abs})=\{\omega \in H^2(\wedge^pT^*\widetilde M); \omega_{\rm nor}=0,(d\omega)_{\rm nor}=0\}$. Let $\lambda_p^{\rm abs}(\widetilde g)=\inf {\rm Spec}(\Delta_p^{\rm abs})$. The spectral argument in \cite[Section 6]{S3} then provides, under the conditions of Theorem \ref{dxesttheo}, 
the lower bound
\begin{equation}\label{mckean}
{\lambda_p^{\rm abs}}(\widetilde g)\geq \frac{1}{4}\left((n-p-1)\sqrt{\kappa}-p\right)^2.
\end{equation} 
We remark that the proof in \cite{S3} uses induction in $p$ starting at $p=0$, which corresponnds to (\ref{schickest}). Here we 
use this to prove the following vanishing result.

\begin{proposition}\label{vanisht}
	Let $(M,g)$ be a compact and convex $\partial$-manifold with infinite fundamental group and assume that $g$ is  $\kappa_p$-negatively pinched where $2\leq 2p<n-1$.  Then $\lambda_p^{\rm abs}(\widetilde g)>0$ and  $(\widetilde M,\widetilde g)$ carries no non-trivial absolute $L^2$ harmonic $p$-form. Hence, $H^p_{(2),\rm abs}(\widetilde M,\widetilde g)$ vanishes.
\end{proposition}

\begin{proof}
	The assumptions imply that $\kappa_p<1$, so we can find $\kappa_p<\kappa< 1$ such that $-1\leq K_{\rm sec}(\widetilde g)\leq -\kappa$. The result follows from (\ref{mckean}) and the $L^2$ Hodge-de Rham theory  in \cite{S2,S3}. 
\end{proof}

\section{The proofs of Theorems \ref{mainesp} and \ref{mainesp2}}\label{proofsmain}

Here we prove the main results of this work. 
Notice that if 
$(r_{(q)},\rho_{(q)})$ is s.s.p. then 
\begin{equation}\label{supfin}
\sup_{x^0\in K}\theta_q(x^0)<+\infty,\,\,\,  {\rm for}\,\,\,  {\rm any}\,\,\, K.
\end{equation}
Also, if $(\alpha,\beta)$ is s.s.p. then $(\overline\alpha,\overline{\beta})$ is s.s.p. as well
for any $\overline{\alpha}\geq\alpha$ and $\overline \beta\geq \beta$.

The proof of Theorem \ref{mainesp} goes as follows. If $M$ is convex with respect to a  $\kappa_p$-negatively pinched metric $g_-$ then $H^p_{(2),\rm abs}(\widetilde M,\widetilde g_-)$ vanishes by Proposition \ref{vanisht}. On the other hand, by standard Hodge theory  for compact $\partial$-manifolds \cite{T}, any non-trivial class in $H^p(M;\mathbb R)$ can be represented by a non-trivial absolute harmonic $p$-form with respect to any metric $g_+$ on $M$. The lift of this form to $(\widetilde M,\widetilde g_+)$ defines a non-trivial absolute harmonic $p$-form which is uniformly bounded. Now, if $g_+$ has both $(r_{(p\pm 1)},\rho_{(p-1)})$ s.s.p.  then the corresponding invariants of $\widetilde g_+$ are s.s.p. as well, since the property is preserved by passage to covers; see Remark \ref{functor}. In particular, (\ref{supfin}) holds with $q=p\pm 1$. Thus we may apply Proposition \ref{import} to conclude that $H^p_{(2),\rm abs}(\widetilde M,\widetilde g_+)\neq \{0\}$. Since $H^p_{(2),\rm abs}(\widetilde M,\cdot)$    is a quasi-isometric invariant of the metric \cite{S3} we obtain a contradiction which completes the proof.

We now consider the case $p=1$ in Theorem \ref{mainesp2}. For its proof we need an extension of a  well-known  result in \cite{LS} to our setting. 

\begin{proposition}
	\label{lyonssull}
	If $(M,g)$ is a compact $\partial$-manifold and $\pi_1(M)$ is nonamenable then $(\widetilde M,\widetilde g)$ carries a nonconstant bounded  absolute harmonic function.   
\end{proposition}

\begin{proof}
	The argument in \cite[Section 5]{LS} carries over to our case. More precisely, using the Neumann heat kernel we construct a natural $\pi_1(M)$-equivariant projection from $L^\infty_{\rm abs}(\widetilde M)$, the space of absolute bounded functions, onto $\mathcal H^\infty_{\rm abs}(\widetilde M,\widetilde g)$, the space of bounded absolute harmonic functions. Also, there exists a $\pi_1(M)$-equivariant injection $l^\infty(\pi_1(M))\hookrightarrow L^\infty_{\rm abs}(\widetilde M)$. Hence, if  $\mathcal H^\infty_{\rm abs}(\widetilde M,\widetilde g)=\mathbb R$ the composition $l^\infty(\pi_1(M))\to\mathbb R$ defines an invariant mean.
\end{proof}

We can now finish the proof of Theorem \ref{mainesp2}. If $M$ carries a $\kappa_1$-negatively curved metric $g_-$ then  $H^1_{(2),\rm abs}(\widetilde M,\widetilde g_-)$ vanishes.  On the other hand, by Proposition \ref{lyonssull}, for {any} metric $g_+$ on $M$, $(\widetilde M,\widetilde g_+)$ carries a 
nonconstant bounded  absolute harmonic function, say $f$. This implies that reflecting Brownian motion in $(\widetilde M,\widetilde g_+)$ is transient and in particular  there holds
	\[
	\sup_{x^0\in K}\int_0^{+\infty} (P_td^\star\phi)(x^0)\,dt<+\infty,
	\] 
for any $K\subset\widetilde M$ and compactly supported $1$-form $\phi$ as in Proposition \ref{keyprop}; see  \cite[Theorem 5.1]{G}. Assuming that $g_-$ is such that 
the corresponding pair $(r_{(2)},\rho_{(2)})$ is s.s.p. we can apply Proposition \ref{import} because $\psi=df$ is a bounded absolute harmonic $1$-form. Thus, $H^1_{(2),\rm abs}(\widetilde M,\widetilde g_+)\neq \{0\}$ and we  get a contradiction. This completes the proof of Theorem \ref{mainesp2}.

\section{A Feynman-Kac formula on $\partial$-man\-ifolds}\label{appA}

In this final section we explain how the method put forward in \cite{A, H1} can be adapted to prove a Feynman-Kac-type formula for $q$-forms on $\partial$-manifolds. 
As an illustration of the flexibility of the method we also include a similar formula for spinors evolving by the heat semigroup of the Dirac Laplacian on ${{\rm spin}^c}$ $\partial$-manifolds. 
These results are presented in subsections \ref{appA2} and \ref{appA3}, respectively, after some preparatory material in subsection \ref{appA1}.

\subsection{The Eells-Elworthy-Malliavin approach}\label{appA1}

Let $(N,h)$ be a Riemannian $\partial$-man\-ifold of dimension $n$. As in  Section \ref{setup} we assume that $(N,\partial N,h)$ has bounded geometry. Let $\pi:P_{{\rm O}_n}(N)\to N$ be the orthonormal frame bundle of $N$. This is a principal bundle with structural group ${\rm O}_n$, the orthogonal group in dimension $n$. Any representation $\zeta:{\rm O}_n\to {\rm End}(V)$  gives rise to the associated vector bundle $\mathcal E_{\zeta}=P_{{\rm O}_n}(N)\times_\zeta V$, which comes endowed with a natural metric and compatible connection derived from $h$ and its Levi-Civita connection $\nabla$. Moreover, any section $\sigma\in\Gamma(\mathcal E_\zeta)$ can be identified to its lift $\sigma^\dagger:P_{{\rm O}_n}(N)\to V$,  which is $\zeta$-equivariant in the sense that $\sigma^\dagger(ug)=\zeta(g^{-1})(\sigma^\dagger(u))$, $u\in P_{{\rm O}_n}(N)$, $g\in {\rm O}_n$.  Also, we recall that in terms of lifts, covariant derivation essentially corresponds to Lie differentiation along horizontal tangent vectors.

Any bundle $\mathcal E_\zeta$ as above comes equipped with a second order elliptic operator $\Delta^B=-{\rm tr}_h\nabla^2:\Gamma(\mathcal E_\zeta)\to\Gamma(\mathcal E_\zeta)$, the Bochner Laplacian. Here, $\nabla^2$ is the standard Hessian operator acting on sections. Given an algebraic (zero order) self-adjoint map $\mathcal R\in\Gamma({\rm End}(\mathcal E_\zeta))$ we can form  
the elliptic operator 
\[
\Delta=\Delta^B +\mathcal R
\]
acting on $\Gamma(\mathcal E_\zeta)$. 
Standard results \cite{Ei,S2,S3} imply that the heat semigroup  $P_t=e^{-\frac{1}{2}t\Delta}$  
has the property that, for any $\sigma_0\in L^2\cap L^{\infty}$, $\sigma_t=P_t\sigma_0$ solves the heat equation
\begin{equation}\label{heateq}
\frac{\partial \sigma_t}{\partial t}+\frac{1}{2}\Delta\sigma_t=0,\quad \lim_{t\to 0}\sigma_t=\sigma_0,
\end{equation}
where we eventually impose elliptic boundary conditons in case $\partial N\neq \emptyset$.

An important question concerning us here is whether the solutions of (\ref{heateq}) admit a stochastic representation in terms of Brownian motion on $N$. If $\partial N=\emptyset$ this  problem admits a very elegant solution in great generality and a Feynman-Kac formula is available \cite{E, Gu, H2, H3, Ma, St}. Moreover, this representation permits to estimate the  solutions in terms of the overall expectation of $\mathcal R$ with respect to the diffusion process; see (\ref{feka-1})-(\ref{feka0}) below.  However, in the presence of a boundary it is well-known that  the problem is much harder to handle; see \cite{H1} and the references therein.

Let us assume that  $N$ has a non-empty boundary endowed with an inward unit normal field $\nu$. We first briefly recall how reflecting Brownian motion is defined on $N$. We take for granted that Brownian motion $\{b^t\}$ on $\mathbb R^n$ is defined. 
This is the diffusion process which has half the standard Laplacian $\sum_i\partial_{i}^2$ as generator.
To transplant this to $N$ we   make use of the so-called Eells-Elworthy-Malliavin approach \cite{E, EE, H2, H3, St}. Note  that any $u\in P_{{\rm O}_n}(N)$ defines an isometry $u:\mathbb R^n\to T_xN$, $x=\pi(u)$. Also, the Levi-Civita connection on $TN$ lifts to an Ehresmann connection on $P_{{\rm O}_n}(M)$ which determines fundamental horizontal vector fields $H_i$, $i=1,\cdots,n$. 
As explained in \cite[Chapter 2]{H2}, these elementary remarks naturally lead to an  identification of  semimartingales  on $\mathbb R^n$,  horizontal semimartingales on $P_{{\rm O}_n}(M)$ and semimartingales on $M$. 
Thus, on $P_{{\rm O}_n}(N)$ we may  consider the stochastic differential equation 
\begin{equation}\label{sstoc}
du^t=\sum_{i=1}^nH_i(u^t)\circ db^t_i+\nu^\dagger(u^t)dl^t,
\end{equation} 
which has a unique solution $\{u^t\}$ starting at any initial frame $u^0$. This is a horizontal reflecting Brownian motion on $P_{{\rm O}_n}(N)$ and its projection  $x^t=\pi u^t$ defines reflecting Brownian motion on $N$ starting at $x^0=\pi u^0$. Moreover, $l^t$ is the associated  boundary local time.

\begin{remark}
	\label{functor}{\rm Due to the obvious functorial character of this construction we easily obtain highly desirable prop\-erties of Brownian motion. For instance, if the manifold splits as an isometric product of two other manifolds then its Brownian motion is simply the product of the motions in the factors. In particular, 
	if $N=X\times Y$, where $Y$ is a compact $\partial$-manifold, then Brownian motion in $N$ is transient if and only if the same happens to $X$. Also, if $\widetilde N\to N$ is a normal Riemannian covering then  Brownian motion in $\widetilde N$ projects down to Brownian motion in $N$. From this it is  obvious that a pair $(\alpha,\beta)$ on $(N,\partial N)$ is s.s.p.  if and only if its lift $(\widetilde \alpha,\widetilde \beta)$ on $(\widetilde N,\partial \widetilde N)$ is s.s.p. as well.}
\end{remark}

We now describe how this formalism leads to an elegant approach to Feyn\-man-Kac-type formulas.
Let $\mathcal A\in\Gamma({\rm End}(\mathcal E_\zeta|_{\partial N})$ be a pointwise self-adjoint map. In practice, $\mathcal A$ relates to  the zero order piece of the given boundary conditions. In analogy with the boundaryless case, It\^o's calculus suggests to consider the multiplicative functional $M^t\in{\rm End}(V)$ satisfying 
\[
dM^t+M^t\left(\frac{1}{2}\mathcal R^\dagger dt+\mathcal A^\dagger dl^t\right)=0, \quad M^0=I.
\]
Standard results imply that a solution exists along each path $u^t$.
We now apply It\^o's formula to the process $M^t\sigma^\dagger(T-t,u^t)$,  $0\leq t\leq T$, where $\sigma$ is a (time-dependent) section of $\mathcal E_\zeta$. With the help of  (\ref{sstoc}) we obtain  
\begin{eqnarray*}
dM^t\sigma^\dagger(T-t,u^t) & = & \left[M^t\mathcal L_{H}\sigma^\dagger(T-t,u^t),db^t\right] - M^t L^\dagger\sigma^\dagger(T-t,u^t)dt\\
& & \quad +M^t\left(\mathcal L_{\nu^\dagger}-\mathcal A^\dagger\right)\sigma^\dagger(T-t,u^t)dl^t,
\end{eqnarray*}
where $\mathcal L$ is Lie derivative,
\[ 
\left[M^t\mathcal L_{H}\sigma^\dagger(T-t,u^t),db_t\right]_i=\sum_{j=1}^{{\rm dim}\, V}\sum_{k=1}^n M^t_{ij}\mathcal L_{H_k}\sigma_j^\dagger(T-t,u^t)db^t_k,
\]
and 
\[
L^\dagger=\frac{\partial }{\partial t}+\frac{1}{2}\left(\Delta^\dagger_B+\mathcal R ^\dagger\right)
\]
is the lifted heat operator, with 
$\Delta_B^\dagger=-\sum_k\mathcal L_{H_k}^2$ being the horizontal Bochner Laplacian. 
Notice that in case $\partial N=\emptyset$ and $\sigma$ satisfies (\ref{heateq}) the computation gives 
\[
dM^t\sigma^\dagger(T-t,u^t)  =  \left[M^t\mathcal L_{H}\sigma^\dagger(T-t,u^t),db^t\right],
\]
which characterizes $M^t\sigma^\dagger(T-t,u^t)$ as a martingale.  Equating the expectations of this process at $t=0$ and $t=T$ 
yields the celebrated Feynman-Kac formula
\begin{equation}\label{feka-1}
\sigma^\dagger(t,u^0)=\mathbb E_{u^0}\left(M^t\sigma^\dagger(0,u^t)\right),
\end{equation}
where $dM^t=-M^t\mathcal R^\dagger dt/2$ \cite{E, H2, H3, Gu, St}. From this we easily obtain the well-known  estimate 
\begin{equation}\label{feka0}
|\sigma(t,x^0)|\leq\mathbb E_{x^0}\left(|\sigma(0,x^t)|\exp\left(-\frac{1}{2}\int_0^t\underline{\mathcal R}(x^s)ds\right)\right),
\end{equation}
where $\underline{\mathcal R}(x)$ is the least eigenvalue of $\mathcal R(x)$. However, if $\partial N\neq\emptyset$ the calculation merely  says
that  $\{M^t\}$ is the multiplicative functional associated with the operator $L$ under  boundary conditions
\begin{equation}\label{boundcond}
(\nabla_{\nu}-\mathcal A)\sigma=0.
\end{equation}
As we shall see below through examples, (\ref{boundcond}) is too stringent to encompass  boundary conditions commonly occurring in applications. 

\subsection{The Feynman-Kac formula for absolute differential forms}\label{appA2}

It turns out that natural elliptic boundary conditions do not quite fit into the  prescription 
in (\ref{boundcond}) essentially because  they are formulated in terms of projections. Hence, the formalism in the previous subsection does not apply as presented. We illustrate this  issue by considering  the case $\zeta=\wedge^q\mu_n^*$, where $\mu_n$ is the birth certificate representation of ${\rm O}_n$, so that $\mathcal E_\zeta$ is the bundle of $q$-forms over $N$. In this case, $\mathcal A$ is explicitly described in terms of the second fundamental form of $\partial N$ but degeneracies occur due to the splitting of forms into tangential and normal components which is inherent to absolute boundary conditions.

The splitting is determined by the ``fermionic relation'' $\nu\righthalfcup\nu\wedge+\nu\wedge\nu\righthalfcup=I$, which induces an orthogonal decomposition  
\[
\wedge^qT^*N|_{\partial N}={\rm Ran}(\nu\righthalfcup\nu\wedge)\oplus{\rm Ran}(\nu\wedge\nu\righthalfcup),
\]
and we denote by $\Pi_{\rm tan}$ and $\Pi_{\rm nor}$ the orthogonal projections onto the factors. As it is clear from the notation, these maps project  onto the space of tangential and normal $q$-forms, respectively.

Let $A:T\partial N\to T\partial N$, $AX=-\nabla_X\nu$, be the second fundamental form of $\partial N$,  which we extend to $TN|_{\partial N}$ by declaring that $A\nu=0$. This induces the pointwise self-adjoint  map $\mathcal A_{q}\in{\rm End}(\wedge^qT^*N|_{\partial N})$, 
\[
(\mathcal A_{q}\omega)(X_1,\cdots,X_q)=
\sum_i\omega(X_1,\cdots,AX_i,\cdots,X_q).
\] 
Notice that $\Pi_{\rm nor}\mathcal A_q\omega=0$, that is,  $\mathcal A_q\omega$ only has tangential components. In order to determine the tangential coefficients of $\mathcal A_q\omega$ we fix an orthonormal frame  $\{e_1,\cdots,e_{n-1}\}$ in $T\partial N$ 
which is principal at $x\in\partial N$ in the sense that $Ae_i=\rho_ie_i$. We then find that, at $x$, 
\begin{equation}\label{clear}
(\mathcal A_q\omega)(e_{i_1},\cdots,e_{i_q})=
\left(\sum_{j=1}^q\rho_{i_j}\right)\omega(e_{i_1},\cdots,e_{i_q}).
\end{equation}

The next result is inspired by \cite[Lemma 4.1]{H1}; see also \cite{Y,DL} for similar computations.

\begin{proposition}
	\label{liftboundcond}
	A $q$-form $\omega$ is absolute if and only if its lift $\omega^\dagger$ satisfies
	\begin{equation}\label{degen}
	 \Pi^\dagger_{\rm nor}\omega^\dagger=0\quad{\rm and}\quad\quad\Pi^\dagger_{\rm tan}(\mathcal L_{\nu^\dagger}-\mathcal A_q^\dagger)\omega^\dagger=0\quad {\rm on}\quad \partial P_{{\rm O}_n}(N).
	\end{equation}
\end{proposition}

\begin{proof}
We work downstairs on $\partial N$ and drop the dagger from the notation. First, $\omega_{\rm nor}=0$ means that  $\omega=\omega_{\rm tan}+\nu\wedge\omega_{\rm nor}=\omega_{\rm tan}$, 
that is, $\Pi^\dagger_{\rm nor}\omega^\dagger=0$. On the other hand, in terms of the principal frame $\{e_i\}$ above,  
\begin{eqnarray*}
\nu\righthalfcup d\omega(e_{i_1},\cdots, e_{i_q})& = & d\omega(\nu,e_{i_1},\cdots, e_{i_q})\\
& = & \nu(\omega(e_{i_1},\cdots, e_{i_q}))+\sum_j(-1)^je_{i_j}(\omega(\nu,e_{i_1},\cdots,\widehat{e_{i_j}},\cdots,e_{i_q}))\\
& & \quad +\sum_j(-1)^j\omega([\nu,e_{i_j}],e_{i_1},\cdots,\widehat{e_{i_ j}},\cdots,e_{i_q})\\
& & \quad\quad +\sum_{1\leq j<k}(-1)^{j+ k}\omega([e_{i_j},e_{i_k}],\nu,e_{i_1},\cdots,\widehat{e_{i_j}},\cdots,\widehat{e_{i_k}},\cdots,e_{i_q})\\
&= & \nu(\omega(e_{i_1},\cdots, e_{i_q}))+\sum_j(-1)^je_{i_j}((\nu\righthalfcup\omega)(e_{i_1},\cdots,\widehat{e_{i_j}},\cdots,e_{i_q}))\\
& & \quad -\sum_j\omega(e_{i_1},\cdots,[\nu,e_{i_j}],\cdots,e_{i_q})\\
&= & \nu(\omega(e_{i_1},\cdots, e_{i_q}))
-\sum_j\omega(e_{i_1},\cdots,\nabla_\nu e_{i_j},\cdots,e_{i_q})\\
& & \quad -\left(\sum_j\rho_{i_j}\right)\omega(e_{i_1},\cdots,e_{i_q}),
\end{eqnarray*}
where we used that  $[e_{i_j},e_{i_k}]=0$, certainly a justifiable assumption, and $\nu\righthalfcup\omega=0$.
But
\[
\nu(\omega(e_{i_1},\cdots, e_{i_q}))=(\nabla_\nu\omega)(e_{i_1},\cdots,e_{i_q})+
\sum_j\omega(e_{i_1},\cdots,\nabla_{\nu} e_{i_j},\cdots,e_{i_q}),
\]
so we obtain 
\[
\nu\righthalfcup d\omega(e_{i_1},\cdots, e_{i_q})=\left(\nabla_\nu-\sum_j\rho_{i_j}\right)\omega(e_{i_1},\cdots, e_{i_q}).
\]
The results follows in view of (\ref{clear}).
\end{proof}

This proposition confirms the inevitable emergence of degeneracies in the context of absolute boundary conditions due to the projections. To remedy this we proceed as in \cite{H1}. We can express the boundary condition as the superposition of two independent components, namely,
\[
\Pi^\dagger_{\rm tan}(\mathcal L_{\nu^\dagger}-\mathcal A_q^\dagger)\omega^\dagger-\Pi^\dagger_{\rm nor}\omega^\dagger=0.
\] 
The key idea, which goes back to \cite{A},  is to fix $\epsilon>0$ and replace $\Pi_{\rm tan}^\dagger$ by $\Pi_{\rm tan}^\dagger+\epsilon { I}$ above, so   
the condition becomes 
\[
\left(\mathcal L_{\nu^\dagger}-\left(\mathcal A^\dagger_q+\epsilon^{-1}\Pi_{\rm nor}^\dagger\right)\right)\omega^\dagger=0,
\] 
which in a sense is the best we can reach in terms of resemblance to (\ref{boundcond}). The next step is to solve for  $\mathcal M^t_\epsilon\in{\rm End}(\wedge^q\mathbb R^n)$ in 
\begin{equation}\label{anteito}
d\mathcal M^t_\epsilon+\mathcal M^t_\epsilon\left(\frac{1}{2}R^\dagger_q(u^t)dt+\left(\mathcal A_q^\dagger(u^t)+\epsilon^{-1}\Pi_{\rm nor}^\dagger(u^t)\right)dl^t\right)=0,\quad \mathcal M^0_\epsilon=I.
\end{equation}

\begin{proposition}
	\label{unifcont} For all $\epsilon>0$ such that $\epsilon^{-1}\geq \underline{\rho}_{(q)}$ we have
	\begin{equation}
	\label{unifcont2}
	|\mathcal M^t_\epsilon|\leq \exp\left(-\frac{1}{2}\int_0^tr_{(q)}(x^s)ds-\int_0^t\rho_{(q)}(x^s)dl^s\right), \quad t>0.
	\end{equation}
\end{proposition}

\begin{proof}
	The same as in \cite[Lemma 3.1]{H1}, once we take into account that, as it is clear from (\ref{clear}), the sums $\sum_{j=1}^q\rho_{i_j}$ are the (possibly non-null) eigenvalues of $\mathcal A_q$.
	\end{proof}
	
The following convergence result provides the crucial input in the argument.

\begin{theorem}
	\label{converg}
	As $\epsilon\to 0$, $\mathcal M^t_\epsilon$ converges to a multiplicative functional $\mathcal M^t$ in the sense that $\lim_{\epsilon\to 0}\mathbb E|\mathcal M^t_\epsilon-\mathcal M^t|^2=0$. Moreover, $\mathcal M^t\Pi_{\rm nor}^\dagger(u)=0$ whenever $u\in\partial O_N$.
\end{theorem} 	

\begin{proof}	
The rather technical proof of this result for $q=1$ is presented in detail in \cite{H1}. Fortunately, with  the formalism above in place, it is not hard to check that the proof of the general case only adds notational difficulties to the original argument. More precisely, in \cite{H1} the letters  $P$ and $Q$ denote normal e tangential projection, respectively. If we replace these symbols by $\Pi_{\rm nor}$ and $\Pi_{\rm tan}$, the proof there works here with minor modifications. Therefore, it is omitted.
\end{proof}

We now have all the ingredients needed to prove the Feynman-Kac-type formula for differential forms.

\begin{theorem}
	\label{fkac}
	Let $\omega_0$ be an absolute $L^2$ $q$-form on $N$ as above. If $P_t=e^{-\frac{1}{2}t\Delta^{\rm abs}_q}$ is the corresponding heat semigroup, so that $\omega_t=P_t\omega_0$ provides the solution to 
	\begin{equation}\label{solinit}
	\frac{\partial \omega_t}{\partial t}+\frac{1}{2}\Delta^{\rm abs}_q\omega_t=0,\quad \lim_{t\to 0}\omega_t=\omega_0, \quad \nu\righthalfcup\omega_t=0,\quad \nu\righthalfcup d\omega_t=0, 
	\end{equation}
	then the following Feynman-Kac formula holds: 
	\begin{equation}
	\label{fkac2}
	\omega^\dagger_t(u^0)=\mathbb E_{u^0}\left(\mathcal M^t\omega_{0}^\dagger(u^t)\right),
	\end{equation} 
	where $u_t$ is the horizontal reflecting Brownian motion starting at $u_0$. As a consequence, 
	\begin{equation}
	\label{controlsol}
	|\omega_t(x^0)|\leq\mathbb E_ {x^0}\left(|\omega_0({x^t})|\exp\left(-\frac{1}{2}\int_0^tr_{(q)}(x^s)ds
	-\int_0^t\rho_{(q)}(x^s)dl^s\right)\right),
	\end{equation}
	where $x^t=\pi u^t$.
\end{theorem}

\begin{proof}
It\^o's formula  and (\ref{sstoc}) yield 
\begin{eqnarray*}
		d\mathcal M^t_\epsilon\omega^\dagger_{T-t}(u^t) & = & 
		\left[\mathcal M^t_\epsilon\mathcal L_{H}\omega^\dagger_{T-t}(u^t),db^t\right]
		- \mathcal M^t_\epsilon L^\dagger\omega^\dagger_{T-t}(u^t)dt\\
		& & \quad +\mathcal M^t_\epsilon\left(\mathcal L_{\nu^\dagger}-\mathcal A^\dagger-\epsilon^{-1}\Pi_{\rm nor}^\dagger\right)\omega^\dagger_{T-t}(u^t)dl^t.
\end{eqnarray*} 
If $\omega_t$ is a solution of (\ref{solinit}) then the second term  on the right-hand side drops out. Moreover, by  Proposition \ref{liftboundcond}  the same happens to the term involving $\epsilon^{-1}$. 
Sending $\epsilon\to 0$ we end up with 
	\begin{eqnarray*}
		d\mathcal M^t\omega^\dagger_{T-t}(u^t) & = & 
		\left[\mathcal M^t\mathcal L_{H}\omega^\dagger_{T-t}(u^t),db^t\right]\\
		& & \quad +\mathcal M^t\Pi^\dagger_{\rm tan}\left(\mathcal L_{\nu^\dagger}-\mathcal A^\dagger\right)\omega^\dagger_{T-t}(u^t)dl^t,
	\end{eqnarray*} 	
where the insertion of  $\Pi^\dagger_{\rm tan}$ in the last term is legitimate due to the last assertion in Theorem \ref{converg}.  
By Proposition \ref{liftboundcond} this actually reduces to   	
\[
d\mathcal M^t\omega^\dagger_{T-t}(u^t)  =  \left[\mathcal M^t\mathcal L_{H}\omega^\dagger_{T-t}(u^t),db^t\right],
\]
which shows that $\mathcal M^t\omega^\dagger_{T-t}(u^t)$ is a martingale. Thus,   (\ref{fkac2}) follows by equating the    expectations at $t=0$ and $t=T$. 
Finally, (\ref{controlsol}) follows from (\ref{unifcont2}).
\end{proof}

The estimate (\ref{controlsol}) has many interesting consequences. We illustrate its usefulness by mentioning a semigroup domination result which can be proved as in \cite[Theorem 3A]{ER2}; see also \cite{DL, E, H2, H3} for similar results.

\begin{theorem}
	\label{semidom}
	Let $(N,\partial N, h)$ be as above and assume that $\rho_{(q)}\geq 0$ for some $1\leq q\leq n-1$.
	Then 
	there holds
	\[
	\left|e^{-\frac{1}{2} t\Delta^{\rm abs}_q}(x,y)\right|\leq
	\left(
	\begin{array}{c}
	n \\
	q
	\end{array}
	\right)
	e^{-\frac{1}{2}\underline r_{(q)}t}
	e^{-\frac{1}{2} t\Delta^{\rm abs}_0}(x,y), \quad x,y\in N, t>0,
	\]
	where  $\underline{r}_{(q)}=\inf_{x\in N} r_{(q)}(x)$. 
	In particular, if $\lambda^{\rm abs}_0(h)+\underline r_{(q)}\geq 0$ and $r_{(q)}>\underline r_{(q)}$ somewhere then   
	$N$ carries no non-trivial absolute $L^2$ harmonic $q$-form. 
\end{theorem}

\subsection{A digression: the Feynman-Kac formula for spinors}\label{appA3}

Let  $N$ be a ${\rm spin}^c$  $\partial$-manifold \cite{Fr}. As usual we assume that $(N,\partial N,h)$ has bounded geometry. Let $\mathbb SN=P_{{\rm Spin}^c_n}(N)\times_{\zeta} V$ be the ${\rm spin}^c$ bundle of $N$, where $\zeta$ is the complex spin representation. Recall that $P_{{\rm Spin}^c_n}(N)$ is a ${\rm Spin}^c$ principal bundle  double covering $P_{{\rm SO}_n}(N)\times P_{{\rm U}_1}(N)$, where $P_{{\rm U}_1}(N)$ is the ${\rm U}_1$ principal bundle associated to the   
auxiliary complex line bundle $\mathcal F$. After fixing a unitary connection $C$ on $\mathcal F$, the Levi-Civita connection  on $TN$ induces a  metric  connection on $\mathbb SN$, still denoted $\nabla$. The corresponding Dirac operator $D:\Gamma(\mathbb SN)\to\Gamma(\mathbb SN)$ is locally given by
$$
D\psi=\sum_{i=1}^n \gamma(e_i)\nabla_{e_i}\psi,\quad \psi\in\Gamma(\mathbb SN),
$$
where $\{e_i\}_{i=1}^n$ is a local orthonormal frame and $\gamma:{\rm Cl}(TN)\to{\rm End}(\mathbb SN)$ is the Clifford product.
The  Dirac Laplacian operator is
\begin{equation}\label{lich}
D^2\psi=\Delta_B\psi+\mathfrak R\psi, 
\end{equation}
where 
\[
\mathfrak R\psi=\frac{R}{4}\psi+\frac{1}{2}\gamma(i\Omega).
\]
Here, $R$ is the scalar curvature of $h$ and $i\Omega$ is the curvature $2$-form of $C$.

The ${{\rm spin}^c}$ bundle $\mathbb SN|_{\partial N}$, obtained by restricting  $\mathbb SN$ to $\partial N$,  becomes a Dirac bundle if its Clifford product is
$$
\gamma^{\intercal}(X)\psi=\gamma(X)\gamma(\nu) \psi,\quad X\in \Gamma(T\partial N), \quad \psi\in \Gamma(\mathbb SN|_{\partial N}),
$$
and its connection is
\begin{equation}\label{conn0}
\nabla^{\intercal}_X\psi  =  \nabla_X\psi-\frac{1}{2}\gamma^{\intercal}(AX)\psi,
\end{equation}
where as usual $A=-\nabla\nu$ is the second fundamental form of $\partial N$; see \cite{NR} and the references therein.
The corresponding Dirac operator $D^{\intercal}:\Gamma(\mathbb SN|_{\partial N})\to\Gamma(\mathbb SN|_{\partial N})$ is
$$
D^{\intercal}\psi=\sum_{j=1}^{n-1}\gamma^{\intercal}(e_j)\nabla^{\intercal}_{e_j}\psi,
$$
where the frame has been adapted so that $e_n=\nu$.
Imposing that $Ae_j=\rho_je_j$, where $\rho_j$ are the principal curvatures of $\partial N$, a direct computation shows that 
$$
D^{\intercal}\psi=\frac{K}{2}\psi+
\sum_{j=1}^{n-1}\gamma(e_j)\nabla_{e_j}\psi,
$$
where $K={\rm tr}\,A$ is the mean curvature. It follows that this tangential Dirac operator enters into the boundary decomposition of $D$, namely, 
\begin{equation}\label{dirt1}
-\gamma(\nu)D=\nabla_\nu+D^{\intercal}-\frac{K}{2}, 
\end{equation}
which by its turn appears in  Green's formula
for the Dirac Laplacian
\begin{equation}\label{dirt2}
\int_N\langle D^2\psi,\xi\rangle dN=\int_N\langle D\psi,D\xi\rangle dN
-\int_{\partial N}\langle \gamma(\nu)D\psi,\xi\rangle d\partial N,
\end{equation}
where $\psi$ and $\xi$ are compactly supported. Also, since $\gamma^\intercal(e_j)\gamma(\nu)=-\gamma(\nu)\gamma^\intercal(e_j)$ and $\nabla_{e_j}^\intercal\gamma(\nu)=\gamma(\nu)\nabla_{e_j}^\intercal$, we see that 
\begin{equation}\label{dirt3}
D^\intercal\gamma(\nu)=-\gamma(\nu)D^\intercal.
\end{equation}

Now fix a nontrivial orthogonal  projection $\Pi\in\Gamma({\rm End}(\mathbb SN|_{\partial N}))$ and set $\Pi_+=\Pi$ and $\Pi_-={I}-\Pi$. It is clear from (\ref{dirt1}) and (\ref{dirt2}) that any of the  boundary conditions
\begin{equation}
\label{bdconddir}
\Pi_\pm\psi=0, \quad \Pi_\mp\left(\nabla_\nu+D^\intercal-\frac{K}{2}\right)\psi=0,
\end{equation}
turns the Dirac Laplacian $D^2$ into a formally self-adjoint operator. 
The 
next definition isolates a notion of compatibility between  the tangential Dirac operator and the projections which will allow us to get rid of the middle term in the second condition above. 

\begin{definition}
	\label{compcond} We say that the tangential Dirac operator $D^\intercal$ intertwines the projections if there holds 
	$\Pi_\pm D^\intercal=D^\intercal\Pi_{\mp}$.
\end{definition}

\begin{remark}
	\label{intetw}
{\rm 
If $D^\intercal$ intertwines the projections then  $\Pi_\pm D^\intercal\Pi_\pm=D^\intercal\Pi_{\mp}\Pi_\pm=0$. Equivalently, $\langle D^\intercal\Pi_{\pm}\psi,\Pi_\pm\xi \rangle=0$ for any spinors $\psi$ and $\xi$.}  
\end{remark}

\begin{proposition}\label{liftdir}
	Under the conditions above assume further that $D^\intercal$ intertwines the projections as in Definition \ref{compcond}. Then a spinor $\psi\in\Gamma(\mathbb SN|_{\partial N})$ satisfies the boundary conditions (\ref{bdconddir}) if and only if its lift $\psi^\dagger:P_{{\rm Spin}^c_n}(N)\to V$ satisfies
	\begin{equation}\label{bdconddir2}
	\Pi_\pm ^\dagger\psi^\dagger=0\quad {\rm and}\quad \Pi^\dagger_\mp\left(\mathcal L_{\nu^\dagger}-\frac{K^\dagger}{2}\right)\psi^\dagger=0\quad {\rm on}\quad \partial P_{{\rm Spin}^c_n}(N).
	\end{equation}
	\end{proposition}
	
	\begin{proof}
	Obvious in view of (\ref{bdconddir}) and Remark \ref{intetw}.
	\end{proof}

We can now proceed exactly like in the previous subsection. 
We assume  that (\ref{bdconddir2}) gives rise to a self-adjoint elliptic realization of $D^2$  and we denote by $e^{-\frac{1}{2}tD^2}$ the corresponding heat semigroup \cite{Gru}.
We lift everything in sight to $P_{{\rm Spin}^c_n}(N)$ and consider there the functional $\mathcal M_\epsilon^t$ defined by 
\[
d\mathcal M^t_\epsilon+\mathcal M^t_\epsilon\left(\frac{1}{2}\mathfrak R^\dagger(u^t) dt+\left(\frac{1}{2}K^\dagger(u^t)+
\epsilon^{-1}\Pi^\dagger_+(u^t)\right)dl^t\right)=0,\quad \mathcal M^0_\epsilon=I.
\]
The limiting functional $\mathcal M^t$, whose existence is guaranteed by the analogue of Theorem \ref{converg}, appears in the corresponding Feynman-Kac formula. 

\begin{theorem}
	\label{fkacdir}
	Let $\psi_0\in\Gamma(\mathbb SN)$ be a spinor satisfying any of the boundary conditions (\ref{bdconddir}), where we assume that $D^\intercal$ intertwines the projections as in Definition \ref{compcond}. If $\psi_t=e^{-\frac{1}{2}tD^2}\psi_0$ is the solution to 
	\begin{equation}\label{solinitdir}
	\frac{\partial \psi_t}{\partial t}+\frac{1}{2}D^2\psi_t=0,\quad \lim_{t\to 0}\psi_t=\psi_0, \quad \Pi_\pm\psi_t=0,\quad \Pi_\mp\left(\nabla_\nu-\frac{K}{2}\right)\psi_t=0, 
	\end{equation}
	then the following Feynman-Kac formula holds: 
	\begin{equation}
	\label{fkac21}
	\psi^\dagger_t(u^0)=\mathbb E_{u^0}\left(\mathcal M_t\psi_{0}^\dagger(u^t)\right),
	\end{equation} 
	where $u^t$ is the horizontal reflecting Brownian motion on $P_{{\rm Spin}^c_n}(N)$ starting at $u^0$. As a consequence, 
	\begin{equation}
	\label{controlsol2}
	|\psi_t(x^0)|\leq\mathbb E_ {x^0}\left(|\psi_0({x^t})|\exp\left(-\frac{1}{2}\int_0^t {\mathfrak r}(x^s)ds-\frac{1}{2}\int_0^t K(x^s)dl^s\right)\right),
	\end{equation}
	where ${\mathfrak r}(x)=\inf_{|\psi|=1}\langle {\mathfrak R}(x)\psi,\psi\rangle$.
\end{theorem}

\begin{proof}
	The same as in Theorem \ref{fkac}.
\end{proof}

It is worthwhile to state the analogue of Theorem \ref{semidom} for spinors.

\begin{theorem}
	\label{semidomsp}
	Let $(N,h)$ be a ${{\rm spin}^c}$ $\partial$-manifold as above and  assume that $K\geq 0$ along $\partial N$. Let $e^{-\frac{1}{2}tD ^2}$ be the heat semigroup of the Dirac Laplacian acting on spinors subject to boundary conditions as in Theorem \ref{fkacdir}.  Then there holds 		\[
	\left|e^{-\frac{1}{2} tD^2}(x,y)\right|\leq
	2^{\left[\frac{n}{2}\right]+1}
	e^{-\frac{1}{2}\underline {\mathfrak r}t}
	e^{-\frac{1}{2} t\Delta^{\rm abs}_0}(x,y), \quad x,y\in N, t>0,
	\]
	where
	$\underline{\mathfrak r}=\inf_{x\in N}{\mathfrak r}(x)$. 
	In particular, if $\lambda^{\rm abs}_0(h)+\underline{\mathfrak r}\geq 0$ and ${\mathfrak r}>\underline{\mathfrak r}$ somewhere then   
	$N$ carries no non-trivial $L^2$ harmonic spinor satisfying the given boundary conditions. 
\end{theorem}

We now discuss a couple of examples of local boundary conditions for spinors to which Theorem \ref{fkacdir} applies. 

\begin{example}
	\label{chiral}{\rm (Chirality boundary condition)
A chilarity  operator on a ${\rm spin}^c$ $\partial $-manifold
$(N,\partial N)$ is an orthogonal  and parallel involution $Q\in\Gamma({\rm End}(\mathbb SN))$ which anti-com\-mutes with Clifford product with any tangent vector. Examples include Clifford product with the complex volume element in an even dimension spin manifold and  with the time-like unit normal to an immersed space-like hypersurface in a Lorentzian spin manifold. It is easy to check that $D^\intercal Q=QD^\intercal$ and $D^\intercal\gamma(\nu)=-\gamma(\nu)D^\intercal$. Given any such $Q$ define the boundary chilarity operator $\widehat Q=\gamma(\nu)Q\in\Gamma({\rm End}(\mathbb SN)|_{\partial N})$, which  still is an orthogonal and parallel involution with associated projections given by 
\begin{equation}\label{projqui}
\Pi_\pm=\frac{1}{2}\left({I}\mp\widehat Q\right).
\end{equation}
Since $D^\intercal \widehat Q=D^\intercal\gamma(\nu)Q=-\gamma(\nu)QD^\intercal=-\widehat QD^\intercal$, we conclude that  $D^\intercal\Pi_{\pm}=\Pi_{\mp}D^\intercal$, that is, $D^\intercal$ intertwines the projections. Thus, Theorem \ref{fkacdir} applies to the self-adjoint elliptic realization of $D^2$ under this boundary condition.}
\end{example} 

\begin{example}
	\label{mit}{\rm (MIT bag boundary condition) This time we choose $\widehat Q=i\gamma(\nu)$, an involution which clearly satisfies $D^\intercal\widehat Q=-\widehat QD^\intercal$. Thus, $D^\intercal$ intertwines the projections exactly like in the previous example and Theorem \ref{fkacdir} again applies to the self-adjoint elliptic realization of $D^2$ under this boundary condition.}
\end{example}

\begin{remark}
	\label{clarif} {\rm For the sake of comparison, it is instructive to examine how absolute and relative boundary conditions for differential forms fit into the framework developed in this subsection. In particular, this helps to clarify the role played by Proposition \ref{liftboundcond} and its analogue for relative forms. Recall that $\wedge^\bullet T^*N$ has the  structure of a Clifford module if we define the Clifford product by tangent vectors as $\gamma(v)=v\wedge-v\righthalfcup$. The corresponding Dirac operator is $D=d+d^\star$, so that $D^2=\Delta$, the Hodge Laplacian. If 
	$\omega$ is a $q$-form then we know that along $\partial N$,
	\[
	\omega=\omega_{\rm tan}+\nu\wedge\omega_{\rm nor}=\Pi_{\rm tan}\omega+\Pi_{\rm nor}\omega.
	\]	
	Instead of (\ref{conn0}) we now have
	\[
	\nabla^\intercal_X=\nabla^{\partial N}_X+\nu\wedge A(X)\righthalfcup.
	\]
	A direct computation then shows that, with respect to the splitting above, the boundary decomposition of $D$ is 
	\[
	-\gamma(\nu)D\left(
	\begin{array}{c}
	\omega_{\rm tan}\\
	\omega_{\rm nor}
	\end{array}
	\right)=
	\left(
	\begin{array}{c}
	\nabla_\nu\omega_{\rm tan}\\
	\nabla_\nu\omega_{\rm nor}
	\end{array}
	\right)
	-\left(
	\begin{array}{cc}
	\mathcal S_q^{\rm tan} & D_{\partial N}\\
	D_{\partial N}& \mathcal S_{q-1}^{\rm nor}
	\end{array}
	\right)
	\left(
	\begin{array}{c}
	\omega_{\rm tan}\\
	\omega_{\rm nor}
	\end{array}
	\right)
	,
	\] 
	where $D_{\partial N}=d_{\partial N}+d^*_{\partial N}$ and in terms of a principal frame,
	\[
	\mathcal S_q^{\rm tan,\,nor}=\sum_j\rho_j\Pi_{e_j}^{\rm tan,\,nor},
	\]
	with $\Pi_{v}^{\rm tan}=v\wedge v\righthalfcup$
	and $\Pi_{v}^{\rm nor}=v\righthalfcup v\wedge$. 
	If $\omega_{\rm nor}=0$ then $\mathcal S_q^{\rm tan}\omega=\mathcal S_q\omega$ and the boundary integral in Green's formula for the Hodge Laplacian is 
	\[
	\int_{\partial N}\left(\langle\nabla_\nu\omega_{\rm tan},\omega_{\rm tan}\rangle-\langle\mathcal S_q\omega_{\rm tan},\omega_{\rm tan}\rangle-\langle D_{\partial N}\omega_{\rm tan},\omega_{\rm tan}\rangle\right)d\partial N.
	\]
	However, the last term vanishes because 
    the forms involved in the inner product have different parities. Thus, the right boundary conditions are
	\begin{equation}\label{absolc}
	\Pi_{\rm nor}\omega =0, \quad \Pi_{\rm tan}\left(\nabla_\nu-\mathcal S_q\right)\omega=0.
	\end{equation}
	Proposition \ref{liftboundcond} then shows that (\ref{absolc}) defines absolute boundary conditions for the Hodge Laplacian. Similarly, if $\omega_{\rm tan}=0$ then $\omega=\nu\wedge \omega_{\rm nor}$ and  $\mathcal S_{q-1}^{\rm nor}\omega_{\rm nor}=\star\mathcal S_{n-q}\star\omega_{\rm nor}$, where here $\star$ is the Hodge star operator of $\partial N$. This time 
	the boundary integral is 
	\[
	\int_{\partial N}\left(\langle\nabla_\nu\omega_{\rm nor},\omega_{\rm nor}\rangle-\langle\star\mathcal S_{n-q}\star\omega_{\rm nor},\omega_{\rm nor}\rangle-\langle D_{\partial N}\omega_{\rm nor},\omega_{\rm nor}\rangle\right)d\partial N.
	\]
	Again, the last term drops out and the correct 
	boundary conditions are
	\begin{equation}\label{relbx}
	\Pi_{\rm tan}\omega =0, \quad \Pi_{\rm nor}\left(\nabla_\nu-\star\mathcal S_{n-q}\star\right)\omega=0. 
	\end{equation}
	As in Proposition \ref{liftboundcond} we compute that 
	\begin{eqnarray*}
	\left(\nabla_\nu-\star \mathcal  S_{n-q}\star\right)\omega(\nu,e_{i_1},\cdots,e_{i_ {q-1}}) & = & (\nu\wedge d^\star\omega)(\nu,e_{i_1},\cdots,e_{i_ {q-1}})\\
	& = & (\nu\righthalfcup\nu\wedge d^\star\omega)(e_{i_1},\cdots,e_{i_ {q-1}})\\
	& = & \left(\Pi_{\rm tan}d^\star\omega\right)(e_{i_1},\cdots,e_{i_ {q-1}})
	\end{eqnarray*}
	so that (\ref{relbx})
can be rewritten as 
\[
\omega_{\rm tan}=0, \quad (d^\star\omega)_{\rm tan}=0.
\]
This is exactly how relative 
boundary conditions for the Hodge Laplacian are defined \cite{T}. We thus see that for differential forms the cancellations leading to the correct boundary conditions are caused by the fact that $D_{\partial N}$ clearly intertwines the projections onto the spaces of even and odd degree forms; compare to Definition \ref{compcond}. }
\end{remark}


\end{document}